\documentclass[journal,10pt,twocolumn,oneside]{IEEEtran}
\usepackage{enumerate}
\usepackage{float}
\usepackage{amssymb,amsmath,amsxtra,amsfonts,amsthm,mathtools}
\usepackage{graphicx}
\usepackage{bigints}
\usepackage{fancyhdr}
\usepackage{float}
\usepackage{times}
\usepackage{wrapfig}
\usepackage{xcolor,colortbl}
\usepackage{soul}
\usepackage{enumitem}
\usepackage{array,multirow,makecell}
\usepackage{subfig}
\usepackage{nopageno}
\usepackage{tikz}
\usepackage{glossaries}
\usepackage{hyperref}
\usepackage[switch,columnwise]{lineno} 
\usepackage{mathrsfs}
\usepackage{lineno}
\usepackage[
backend=biber,
style=numeric,
sorting=ynt
]{biblatex}

\usepackage{amsthm}

\newtheorem{assumption}{Assumption}
\newtheorem{definition}{Definition}
\newtheorem{remark}{Remark}
\newtheorem{lemma}{Lemma}
\newtheorem{theorem}{Theorem}
\newtheorem{corollary}{Corollary}

\newcommand\norm[1]{\left\lVert#1\right\rVert}

\DeclareUnicodeCharacter{0301}{\'}

\begin{document}
\title{\textbf{Probabilistic Verification of Approximate Algorithms with Unstructured Errors: Application to Fully Inexact Generalized ADMM}}

\author{
  Anis Hamadouche,\,\, 
  Yun Wu,\,\,
  Andrew M.\ Wallace,\,\,
  and\,
  Jo\~ao F.\ C.\ Mota%
  \IEEEcompsocitemizethanks{
    \IEEEcompsocthanksitem
    Work supported by UK's EPSRC (EP/T026111/1, EP/S000631/1), and the MOD University Defence Research Collaboration. 
    \IEEEcompsocthanksitem
    Anis Hamadouche, Yun Wu, Andrew M.\ Wallace, and Jo\~ao F.\ C.\ Mota
    are with the School of Engineering \& Physical Sciences, Heriot-Watt University, Edinburgh EH14 4AS,
    UK. (e-mail: \{ah225,y.wu,a.m.wallace,j.mota\}@hw.ac.uk).
  }
}

\maketitle
\thispagestyle{plain}
\pagestyle{plain}

\begin{abstract}
\textbf{We analyse the convergence of an approximate, fully inexact, ADMM algorithm under additive, deterministic and probabilistic error models. We consider the generalized ADMM scheme that is derived from generalized Lagrangian penalty with additive (smoothing) adaptive-metric quadratic proximal perturbations. We derive explicit deterministic and probabilistic convergence upper bounds for the lower-$\mathcal{C}^2$ nonconvex case as well as the convex case under the Lipschitz continuity condition. We also present more practical conditions on the proximal errors under which convergence of the approximate ADMM to a suboptimal solution is guaranteed with high probability. We consider statistically and dynamically-unstructured conditional mean independent bounded error sequences. We validate our results using both simulated and practical software and algorithmic computational perturbations. We apply the proposed algorithm to a synthetic LASSO and robust regression with $k$-support norm regularization problems and test our proposed bounds under different computational noise levels. Compared to classical convergence results, the adaptive probabilistic bounds are more accurate in predicting the distance from the optimal set and parasitic residual error under different sources of inaccuracies.}
\end{abstract}

\IEEEpeerreviewmaketitle
\newcommand\scalemath[2]{\scalebox{#1}{\mbox{\ensuremath{\displaystyle #2}}}}
\section{INTRODUCTION}
The Alternating Direction of Multiplier Method (ADMM) originally emerged from the early works of Peaceman and Rachford \cite{peaceman1955numerical, douglas1955numerical, douglas1956numerical, douglas1964general}, Glowinski and Marrocco \cite{glowinski1975approximation}, and Gabay and Mercier \cite{gabay1976dual} on applying numerical procedures to solve partial differential equations arising in heat conduction and continuum mechanics problems. Due to its simple implementation and efficiency in solving large-scale optimization problems encountered in statistics, and more recent machine learning problems, ADMM has evolved into different forms and has been applied to solve more general nonconvex problems that cannot be solved using conventional methods. ADMM is also a popular method for online and distributed optimization \cite{boyd2011distributed}. 
ADMM can be viewed as a decomposition procedure that exploit the separability of the objective function to reduce the difficulty of joint minimisation arising from classical methods of multiplier \cite{hestenes1969multiplier, powell1969method,rockafellar1973multiplier, bertsekas1976multiplier} at the cost of increased number of iterations \cite{deng2016global}. However, this procedure yields efficient local solutions to subproblems that are coordinated via constraints to find the global solution \cite{gabay1976dual,glowinski1975approximation}. The subproblems are typically easier to solve than the original problem, and as a matter of fact, closed-form solutions are often available. 
For instance, in the following LASSO problem
\begin{equation}
  \begin{split}
      &\underset{x \in \mathbb{R}^n,z \in \mathbb{R}^m}{\text{minimize}} \,\,\,
      f(x) := \|Ax-b\|_2^2 + \|z\|_1,\\
      &\text{subject to} \,\,\,
      x-z = 0\,,  
  \end{split}
\end{equation}
when the data matrix $A$ has some circulant structure as those in signal/image processing, the system can be solved via fast Fourier transform with insignificant computation cost \cite{afonso2010fast, allison2012accelerated, matakos2013accelerated}. However, when the ADMM subproblems do not possess closed-form solutions, the solution is obtained (or approximated) iteratively which can lead to increased computation cost with additional incurred truncation (computational) errors. This motivates the need for new computationally-aware variants of the classical ADMM with  inexpensive auxiliary subproblems.

Another feature that makes ADMM technique practically appealing is its robustness to poorly selected algorithm parameters, i.e., the method is guaranteed to converge independently from the selected parameter. For instance, if the objective functions are strongly convex and have Lipschitz-continuous gradients, then the iterates produced by the ADMM algorithm converge linearly to the optimum in a certain distance metric no matter which parameter value is selected \cite{deng2016global, giselsson2014diagonal, giselsson2016linear, nishihara2015general,francca2016explicit, hong2017linear, han2018linear, chen2021equivalence}. For an extensive review of ADMM, the interested reader is referred to the survey papers by \cite{boyd2011distributed, eckstein2015understanding, glowinski2014alternating}.

Before we delve into the details of composite optimization, ADMM technique and proximal calculus, let us first recall some basic definitions from the theory of optimization and proximal calculus. 
\subsection{Background}
\begin{definition}[Convex and Concave functions]
A function $f: \mathbb{R}^n\rightarrow \mathbb{R}\cup\{\pm\infty\}$ is \textit{convex} if 
\begin{equation}
    f(\beta x + (1-\beta)y) \leq \beta f(x) + (1-\beta) f(y),
    \quad \forall x,y \in \mathbb{R}^n.
\end{equation}
with $\beta \in (0,1)$. If $f$ is convex then $-f$ is \textit{concave}.
\end{definition}
\begin{definition}[Proper function]
A function is \textit{proper} of its value is never $-\infty$ and it is finite somewhere.
\end{definition}
\begin{definition}[lower semicontinuous function]
A function $f: \mathbb{R}^n\rightarrow \mathbb{R}\cup\{\pm\infty\}$ is \textit{lower semicontinuous} at a point $y \in \mathbb{R}^n$ if and only if
\begin{equation}
    \underset{x \rightarrow y}{\lim \inf} f(x) \leq f(y)
\end{equation}
\end{definition}
\begin{definition}[Closed function]
A convex and proper function is \textit{closed} if and only if it is lower semicontinuous.
\end{definition}
\begin{definition}[Saddle point]
Let $\mathcal{L}(x,y):\mathbb{R}^n\times \mathbb{R}^m:\rightarrow\mathbb{R}\cup\{\pm\infty\}$ be a convex-concave function. The tuple $(x^\star,y^\star)$ is said to be a \textit{saddle point} of $\mathcal{L}$ if
\begin{equation}
    \mathcal{L}(x^\star,y) \leq
    \mathcal{L}(x^\star,y^\star) \leq
    \mathcal{L}(x,y^\star)
\end{equation}
for all $x\in\mathbb{R}^n$, $y\in\mathbb{R}^m$.
\end{definition}
\begin{definition}[subgradient]
\label{Def:Subgradient}
Let $w \in \partial f(x)$, then $\forall y \in\mathbb{R}^n$,  we have
\begin{equation}
    f(y) \geq f(x) + \langle w , y-x \rangle.
\end{equation}
\end{definition}
\begin{definition}[$\varepsilon$-subgradient]
\label{Def:e-Subgradient}
Let $w \in \partial_{\varepsilon} f(x)$, then $\forall y \in\mathbb{R}^n$, we have
\begin{equation}
    f(y) \geq f(x) + \langle w , y-x \rangle-\varepsilon.
\end{equation}
\end{definition}
\begin{definition}[Proximal operator]
\label{Def:Prox}
For all $y \in \mathbb{R}^n$ we have
\begin{equation}
\begin{split}
  \label{Eq:ProxApproximate}    
    \text{prox}_{u}^{}(y)      :=     \underset{z}{\arg \min}\,\, u(z) + \frac{1}{2}\|z -       y\|_2^2.
\end{split}
\end{equation}
\end{definition}
\begin{definition}[Approximate Proximal operator]
\label{Def:eProx}
\begin{equation}
\begin{split}
  \label{Eq:ProxApproximate}    
    \text{prox}_{u}^{\varepsilon}(y) 
    := 
    \Big\{
      x \in \mathbb{R}^n\,:\,
      u(x) + \frac{1}{2}\|x -y\|_2^2 \leq\\ \varepsilon + \underset{z}{\inf}\,\,
      u(z) + \frac{1}{2}\|z -
      y\|_2^2 
    \Big\}\,,    
\end{split}
\end{equation}
\end{definition}

\begin{definition} [lower-$\mathcal{C}^2$ functions \cite{hare2009computing}]
The function $f$ is \textit{lower-$\mathcal{C}^2$} on an open set $V$ if for each $\overline{x}\in V$ there is a neighbourhood $V_{\overline{x}}$ of $\overline{x}$ upon which a representation $f(x) = \underset{t\in T}{\max}  f_t (x)$ holds for $x$ in $V_{\overline{x}}$, where $T$ is a compact set and the functions $f_t$ are twice continuously differentiable jointly in both $x$ and $t$.
\end{definition}
\begin{definition}[lower-$\mathcal{C}^2$ functions \cite{rockafellar2009variational}]
The function $f$ is \textit{lower-$\mathcal{C}^2$} on an open set $V$ if at any point $x$ in $V$ , the sum of $f$ with a quadratic term is a convex function on an open neighbourhood $V'$ of $x$.
\end{definition}
\begin{definition} Weighted norm
Given a PSD matrix $M$, we say that $\|x\|_{M}^2 = \langle x, M x \rangle = \langle  Mx, x \rangle$ is an M-weighted norm of $x$.
\end{definition}

\subsection{Problem statement}
The category of problems that ADMM solves is
\begin{equation}
  \label{Eq:Problem}
  \begin{split}
      &\underset{x \in \mathbb{R}^n,z \in \mathbb{R}^m}{\text{minimize}} \,\,\,
      f(x) := g(x) + h(z)\,,\\
      &\text{subject to} \,\,\,
      Ax+Bz = c\,,  
  \end{split}
\end{equation}
where $g$ and $h$ are lower-$\mathcal{C}^2$, possibly convex and nondifferentiable functions, $A \in \mathbb{R}^{p\times n}$ and $B \in \mathbb{R}^{p\times m}$. 

The \textit{augmented Lagrangian} for problem \eqref{Eq:Problem} is constructed as follows
\begin{equation}
\label{Lagrangian}
    \mathcal{L}_\rho(x,z, y) = g(x)+h(z) + y^\top (Ax+Bz - c) + \frac{\rho}{2}\norm{Ax+Bz-c}_2^2 
\end{equation}
where $y$ is the Lagrange multiplier associated with the linear constraint in \eqref{Eq:Problem}, and $\rho$ is a positive scalar parameter. The ADMM iteration is obtained by minimizing the augmented Lagrangian \eqref{Lagrangian} with respect to $x$ and $z$ variables, and updating the dual variable $y$ as follows
\begin{subequations}
\label{ADMM0}
    \begin{alignat}{3}
    &x^{k+1} = \underset{x}{\arg\min}\quad \mathcal{L}_\rho(x,z^k, y^k)&\\
    &z^{k+1} = \underset{z}{\arg\min}\quad \mathcal{L}_\rho(x^{k+1},z, y^k)&\\
    &y^{k+1} = y^k+\rho(Ax^{k+1}+Bz^{k+1}-c)&.
    \end{alignat}
\end{subequations}
The equivalent \textit{scaled proximal ADMM} is given by 
\begin{subequations}
\label{ADMMProx}
    \begin{alignat}{3}
    &x^{k+1} = \underset{x}{\arg\min}\quad g(x)+\frac{1}{2\lambda}\|Ax+Bz^k-c+v^k\|^2&\label{ADMMProxSub1}\\
    &z^{k+1} = \underset{z}{\arg\min}\quad h(z)+\frac{1}{2\lambda}\|Ax^{k+1}+Bz-c+v^k\|^2&\label{ADMMProxSub2}\\
    &v^{k+1} = v^k+(Ax^{k+1}+Bz^{k+1}-c)&.
    \end{alignat}
\end{subequations}
where $v^k=({1}/{\rho})y^k$ and $\lambda = 1/\rho$.

Consider the case where both ADMM subproblems \eqref{ADMMProxSub1} and \eqref{ADMMProxSub1} are solved inexactly and/or using iterates with additive perturbations ($r_x^k$ and $r_z^k$); i.e., $x^k = \overline{x}^k + r_x^k$ and $z^k = \overline{z}^k + r_z^k$, where $\overline{x}^k$ and $\overline{z}^k$ are the exact iterates. Classical convergence guarantees cannot be easily extended to the inexact case since they completely ignore the presence and propagation of computational errors during iterations. Therefore, a more rigorous treatment of such errors is needed to verify the convergence, reliability and portability of the ADMM (and its variants) to different computational environments.

Before we analyse the convergence of the inexact ADMM, let us introduce some other variants of the original scheme and focus our analysis to the most generalized form. 
\subsection{WL-ADMM approximation}
The approximate WL-ADMM is formulated by approximating the euclidean norm by some positive definite (PD) $L_k $-weighted norm \big($\|.\|_{L_k }^2$\big) in the $x$ and $z$ updates as follows:
\begin{subequations}
\label{WL-ADMM}
    \begin{alignat}{3}
    \label{WLADMM1}
    &x^{k+1} = \underset{x}{\arg\min}\quad g(x)+\frac{1}{2\lambda}\|Ax+Bz^k-c+v^k\|_{L_k }^2&\\
    \label{WLADMM2}
    &z^{k+1} = \underset{z}{\arg\min}\quad h(z)+\frac{1}{2\lambda}\|Ax^{k+1}+Bz-c+v^k\|_{L_k }^2&\\
    \label{WLADMM3}
    &v^{k+1} = v^k+(Ax^{k+1}+Bz^{k+1}-c)&.
    \end{alignat}
\end{subequations}
where $v^k=({1}/{\rho})y^k$, $\lambda = 1/\rho$.
\subsection{WLM-ADMM approximation}
Adding quadratic perturbations to \eqref{WLADMM1} and \eqref{WLADMM2} yields the approximate WLM-ADMM
\begin{subequations}
\label{WLM-ADMM}
    \begin{alignat}{3}
    &x^{k+1} = \underset{x}{\arg\min}\quad g(x)+\frac{1}{2\lambda}\|Ax+Bz^k-c+v^k\|_{L_k }^2&\notag\\&+\frac{1}{2}\|x-x^k\|_{M_x^k}^2&\\
    &z^{k+1} = \underset{z}{\arg\min}\quad h(z)+\frac{1}{2\lambda}\|Ax^{k+1}+Bz-c+v^k\|_{L_k }^2&\notag\\&+\frac{1}{2}\|z-z^k\|_{M_z^k}^2&\\
    &v^{k+1} = v^k+(Ax^{k+1}+Bz^{k+1}-c)&.
    \end{alignat}
\end{subequations}
where $M_x^k$ and $M_z^k$ are positive definite matrices. Note that in the case of convex function $f$, adding quadratic terms (proximal terms) makes the objective functions strongly convex which improves the condition number of the problem being solved at the expense of yielding approximate solutions. 

\eqref{WLM-ADMM} can be written as
\begin{subequations}
\label{WLM-ADMM1.1}
    \begin{alignat}{3}
    &x^{k+1} = \underset{x}{\arg\min}\quad g(x)+\frac{1}{2}\|\Lambda_{1_k}^{1/2}x -\Lambda_{1_k}^{-1/2}\Gamma_{1_k}\|^2&\\
    &z^{k+1} = \underset{z}{\arg\min}\quad h(z)+\frac{1}{2}\|\Lambda_{2_k}^{1/2}z -\Lambda_{2_k}^{-1/2}\Gamma_{2_k}\|^2&\\
    &v^{k+1} = v^k+(Ax^{k+1}+Bz^{k+1}-c)&.
    \end{alignat}
\end{subequations}
Or in generalized norms form as follows
\begin{subequations}
\label{WLM-ADMMP1.2}
    \begin{alignat}{3}
    &x^{k+1} = \underset{x}{\arg\min}\quad g(x)+\frac{1}{2}\|x -\Lambda_{1_k}^{-1}\Gamma_{1_k}\|_{\Lambda_{1_k}}^2&\\
    &z^{k+1} = \underset{z}{\arg\min}\quad h(z)+\frac{1}{2}\|z -\Lambda_{2_k}^{-1}\Gamma_{2_k}\|_{\Lambda_{2_k}}^2&\\
    &v^{k+1} = v^k+(Ax^{k+1}+Bz^{k+1}-c)&.
    \end{alignat}
\end{subequations}
where 
\begin{subequations}
\label{WLM-ADMM1.2Vars}
    \begin{alignat}{3}
    & \Lambda_{1_k} = \frac{1}{\lambda} A^\top L_k  A + M_x^k&\\
    & \Gamma_{1_k} = M_x^k x^k-\frac{1}{\lambda}A^\top L_k  (Bz^k-c+v^k)&\\
    & \Lambda_{2_k} = \frac{1}{\lambda} B^\top L_k  B + M_z^k&\\
    & \Gamma_{2_k} = M_z^k z^k-\frac{1}{\lambda}B^\top L_k  (Ax^{k+1}-c+v^k)&   
    \end{alignat}
\end{subequations}
\subsection{Related work}
A new ADMM variant of \eqref{ADMM0} was proposed in \cite{han2014customized} that yields simple proximal evaluations of the objective functions according the following scheme 
\begin{subequations}
\label{WLM-ADMM1.1}
    \begin{alignat}{3}
    \label{WLM-ADMM1.1Eq1}
    &\zeta_{x}^k=\frac{1}{\rho}(s_{x}^{k-1}-x^k) \in \partial g(x^k)&\\
    &e_{x}(x^k,v^k)=x^k-P_{\mathcal{X}}[x^k-\rho(\zeta^k-A^\top y^k)]&\\    
    &s_{x}^k=x^k+\rho \zeta_{x}^k-\gamma \alpha_k e_{x}(x^k,y^k)&\\    
    &x^{k+1} = \underset{x}{\arg\min}\quad g(x)+\frac{1}{2\rho}\|x -s_{x}^k\|^2&\\
    \label{WLM-ADMM1.1Eq2}
    &\zeta_{z}^k=\frac{1}{\rho}(s_{z}^{k-1}-z^k) \in \partial h(z^k)&\\    
    &e_{z}(z^k,y^k)=z^k-P_{\mathcal{Z}}[z^k-\rho(\zeta_{z}^k-B^\top y^k)]&\\        
    &s_{z}^k=z^k+\rho \zeta_{z}^k-\gamma \alpha_k e_{z}(z^k,y^k)&\\    
    &z^{k+1} = \underset{z}{\arg\min}\quad h(z)+\frac{1}{2\rho}\|z -s_{z}^k\|^2&\\
    &y^{k+1} = y^k+\rho(Ax^{k+1}+Bz^{k+1}-c).
    \end{alignat}
\end{subequations}
where $P_{\mathcal{X}}$ and $P_{\mathcal{Z}}$ denote the orthogonal projections into subspaces $\mathcal{X}$ and $\mathcal{Z}$, respectively. The stepsize sequence $\{\alpha_k\}_{k > 0}$ is chosen with convergence guarantees. Note that \eqref{WLM-ADMM1.1Eq1} and \eqref{WLM-ADMM1.1Eq2} are both proximal evaluations of $g$ and $h$ at $s_x^k$ and $s_z^k$, respectively.

\subsubsection{The inexact proximal ADMM}
For the method to be implementable, it is important for numerical optimisation algorithms to handle approximate solutions. Considering error-resilient applications which involve optimization problems that do not require highly accurate solutions, more computationally-friendly inexact (approximate) ADMM schemes can be used instead of the original scheme. 

Motivated by Rockafellar's proximal method of multipliers, which involves an augmented Lagrangian with an additional quadratic proximal term, \cite{eckstein1994some, fazel2013hankel, he2002new} proposed inexact \textit{semi-proximal ADMM} with added proximal terms as follows
\begin{subequations}
\label{InexactProxADMM}
    \begin{alignat}{3}
    \label{InexactProxADMMEQ1}
    &x^{k+1} \approx \underset{x}{\arg\min}\quad \mathcal{L}_\rho(x,z^k, y^k)+\frac{1}{2}\|x-x^k\|_{M_x^k}^2&\\
    \label{InexactProxADMMEQ2}
    &z^{k+1} \approx \underset{z}{\arg\min}\quad \mathcal{L}_\rho(x^{k+1},z, y^k)+\frac{1}{2}\|z-z^k\|_{M_z^k}^2&&\\
    &y^{k+1} = y^k+\alpha\rho(Ax^{k+1}+Bz^{k+1}-c)&.
    \end{alignat}
\end{subequations}
with Fortin and Glowinski’s relaxation factor $\alpha\in[0, (1+\sqrt{5})/2]$ proposed by \cite{fazel2013hankel}. This algorithm is shown to preserve the good features of the primal proximal method of multipliers, with the additional advantage that it leads to a decoupling of the constraints. $M_x^k$ and $M_z^k$ are selected to be PSD matrices in \cite{fazel2013hankel} and in \cite{he2002new} they are selected to have some positive minimum spectral radius. Note that the inexact semi-proximal ADMM reduces to the classical ADMM when $M_x^k=M_z^k=0$ and to the proximal ADMM \cite{eckstein1994some} when both proximal matrices are PD and $\alpha=1$. Global convergence of the semi-proximal ADMM was proved in \cite{deng2016global} and the iteration complexity (convergence rate) of $O(1/k)$ was derived in \cite{he20121}. Faster rates can be achieved with varying $\alpha$ and better selection of $\rho$, $M_x^k$ and $M_z^k$ \cite{deng2016global}. For instance, \cite{he2002new,deng2016global,fazel2013hankel,li2015proximal} considered indefinite self-adaptive $M_x^k$ and $M_z^k$ and an optimal self-adaptive scheme is presented in \cite{he2020optimally}. $M_x$ and $M_z$ are typically selected to obtain a closed-form solution of the subproblems \eqref{InexactProxADMMEQ1} and \eqref{InexactProxADMMEQ2}. For instance, if we update
\begin{subequations}
    \begin{alignat}{3}
    &M_x^k = \mu_x^k I-\rho A^\top A\\
    &M_z^k = \mu_z^k I-\rho B^\top B.
    \end{alignat}
\end{subequations}
with $\mu_x^k > \rho \|A^\top A\|$ and $\mu_z^k > \rho \|B^\top B\|$ then the quadratic terms in \eqref{InexactProxADMM} are linearized, and the linearized subproblems admit closed-form solutions in terms of resolvent operators of $\partial g$ and $\partial h$, respectively \cite{li2015proximal}. Relaxing the PSD condition on the proximal matrices $M_x^k$ and $M_z^k$, \cite{deng2016global} proved global convergence under sufficient conditions and derived explicit linear rates for different scenarios assuming strong convexity, a gradients Lipschitz continuity and/or full rank conditions on the matrices $A$ and $B$.

\subsubsection{The generalized ADMM}
Using a variable penalty in the augmented Lagrangian, \cite{kontogiorgis1998variable} proposed the following modified ADMM iterative scheme
\begin{subequations}
\label{ADMM+GenLagrangian}
    \begin{alignat}{3}
    &x^{k+1} \approx \underset{x}{\arg\min}\quad \mathcal{L}_{L_k}(x,z^k, y^k)&\\
    &z^{k+1} \approx \underset{z}{\arg\min}\quad \mathcal{L}_{L_k}(x^{k+1},z, y^k)&\\
    &y^{k+1} = y^k+L_k(Ax^{k+1}+Bz^{k+1}-c)&.
    \end{alignat}
\end{subequations}
where the generalized augmented Lagrangian $\mathcal{L}_{L_k}(x,z, y)$ is given by
\begin{equation}
\label{GenLagrangian}
    \mathcal{L}_{L_k}(x,z, y) = g(x)+h(z) + y^\top L_k(Ax+Bz - c) + \frac{1}{2}\norm{Ax+Bz-c}_{L_k}^2 
\end{equation}
with symmetric positive definite matrices $\{L_k\}$ that satisfy some spectral properties. A relaxed variant of \eqref{ADMM+GenLagrangian} can be obtained by relaxing the multiplier updates as follows
\begin{subequations}
\label{ADMM+GenLagrangian}
    \begin{alignat}{3}
    &x^{k+1} \approx \underset{x}{\arg\min}\quad \mathcal{L}_{L_k}(x,z^k, y^k)&\\
    &z^{k+1} \approx \underset{z}{\arg\min}\quad \mathcal{L}_{L_k}(x^{k+1},z, y^k)&\\
    &y^{k+1} = y^k+\alpha L_k (Ax^{k+1}+Bz^{k+1}-c)&.
    \end{alignat}
\end{subequations}
where $\alpha\in[0, (1+\sqrt{5})/2]$ is the Fortin and Glowinski’s relaxation factor. 
The authors in \cite{he2002new} also  consider \eqref{ADMM+GenLagrangian} for a sequence of bounded PD matrices $\{L_k\}_{k\geq 1}$ and prove the convergence of the semi-proximal Algorithm restricted to $\alpha = 1$ for differentiable functions $f$ and $g$. In \cite{zhang2011unified}, $\alpha$ is replaced by a general positive definite matrix $C$; i.e., the dual update becomes
\begin{equation}
    y^{k+1} = y^k+C (Ax^{k+1}+Bz^{k+1}-c).
\end{equation}
and convergence is established assuming $C$ has a maximum spectral radius of $1$.

By using a constant penalty matrix $L_k = H$ and linearizing the quadratic term $\frac{1}{2}\norm{Ax+Bz-c}_{H}^2$, the authors in \cite{xu2011class} proposed three modifications of the generalised scheme in \eqref{ADMM+GenLagrangian} and established sufficient conditions for their convergence but without deriving explicit rates. 

\begin{remark}
The objective of adding generalised proximal terms \eqref{InexactProxADMM} or using adaptive penalty \eqref{ADMM+GenLagrangian} is to improve the convergence speed of the original algorithm \eqref{ADMM0}. Adding generalised proximal terms \eqref{InexactProxADMM} improves the condition number of the subproblems and can also yield closed-form solutions. Using adaptive penalty matrix \eqref{ADMM+GenLagrangian} can be viewed as performing a constraint preconditioning step, which was shown to improve the convergence speed in quadratic programming \cite{ghadimi2014optimal}.
\end{remark}

\subsubsection{Termination criteria for ADMM}
By exploiting the link between the ADMM, the proximal point algorithm (PPA) \cite{chen1994proximal} and Douglas–Rachford (DR) splitting for maximal monotone operators \cite{douglas1955numerical,douglas1964general}, along with a relative approximation error of the hybrid proximal projection (HPP) \cite{solodov1999hybrid}, \cite{eckstein2018relative} presented convergence guarantees for partially inexact ADMM based on some relative error criterion (in one subproblem); which is an extension of the absolute error criterion of \cite{eckstein1992douglas}. More recently, the weak convergence of relaxed inertial and inexact ADMM variants was established in \cite{alves2019iteration,geremia2020relative,alves2020relative} using same relative error criteria of \cite{eckstein2018relative} but with additional projection correction step  \cite{alvarez2004weak, solodov1999hybrid, eckstein2009general, }. Weak convergence of the fully inexact ADMM with summable error and relative error tolerances can be found in \cite{svaiter2011weak} and \cite{svaiter2019weakly}, respectively. 

In the original paper \cite{eckstein2018relative}, the author derived a partially inexact (only one subproblem is solved inexactly) version of ADMM with the following relative-error termination criterion
\begin{equation}
    \|\nu^{k,l}\| \leq \kappa \|y^{k,l}-y^{k}-\rho(z^{k,l}-z^{k})\| 
\end{equation}
where $\nu^{k,l} \in \partial_{x} \mathcal{L}_{\rho}(x^{k,l},z^k,y^k)$ and $l$ is the iteration counter of the inner loop of the first subproblem of \eqref{ADMM0}. $z^{k,l}$ and $y^{k,l}$ are updated as in \eqref{ADMM0} by using the inner (being improved) iterate $x^{k,l}$ instead of the exact $x^k$. A quadratic variant of this criterion is used in \cite{alves2020relative} with the additional effect of inertia as follows
\begin{equation}
    \|\nu^{k,l}\|^2 \leq \kappa^2 \big[\|y^{k,l}-y^{k}-\rho(z^{k,l}-z^{k})\|^2+\rho^2\|x^{k,l}-z^{k,l}\|^2\big] .
\end{equation}
Although both criteria are applicable, the following criterion manifested better numerical performance \cite{alves2020relative}
\begin{equation}
    \|\nu^{k,l}\|^2 \leq \kappa^2 \big[\|y^{k,l}-y^{k}-\rho(z^{k,l}-z^{k})\|^2+\rho^2\|x^{k,l}-z^{k,l}\|^2\big]. 
\end{equation}
The outer loop termination criterion used in \cite{eckstein2018relative,alves2020relative} is given by
\begin{equation}
    \inf(\|-\partial_x[f(x)+g(x)]_{x=x^k}\|_\infty)\leq \varepsilon;\quad \varepsilon > 0,
\end{equation}
\subsection{Our approach}
In this work, we use the WLM-ADMM variant \eqref{WLM-ADMM} which combines both proximal features of \eqref{InexactProxADMM} and generalized penalty of \eqref{ADMM+GenLagrangian} but with scalar relaxation and with a specific choice of the proximal matrices as to yield proximal operations in the solution of ADMM subproblems as follows
\begin{equation}
    M_x^k = \lambda_x I-\frac{1}{\lambda}A^\top L  A
\end{equation}
and
\begin{equation}
    M_z^k = \lambda_zI-\frac{1}{\lambda}B^\top L  B.
\end{equation}
Substituting $M_x^k$ and $M_z^k$ in \eqref{WLM-ADMM1.2Vars} yields the following \textit{scaled proximal WLM-ADMM} iterative scheme
\begin{subequations}
\label{WLM-ADMMProx}
    \begin{alignat}{3}
    &x^{k+1} = \text{prox}_{\frac{1}{\lambda_x} g}^{\varepsilon_g^{k+1}}(\Gamma_{1_k})\approx \text{prox}_{\frac{1}{\lambda_x} g}^{}(\Gamma_{1_k})\label{subproblem_g}\\
    &z^{k+1} = \text{prox}_{\frac{1}{\lambda_z} h}^{\varepsilon_h^{k+1}}(\Gamma_{2_k}) \approx \text{prox}_{\frac{1}{\lambda_z} h}^{}(\Gamma_{2_k})\label{subproblem_h}\\
    &v^{k+1} = v^k+(Ax^{k+1}+Bz^{k+1}-c).
    \end{alignat}
\end{subequations}
where 
\begin{subequations}
\label{WLM-ADMMProx2Norm}
    \begin{alignat}{3}
    & \Gamma_{1_k} = \Sigma_1 x^k-\frac{1}{\lambda}A^\top L  \bigg(Bz^k-c+v^k\bigg)&\\
    &\Sigma_1 = \bigg(\lambda_x I-\frac{1}{\lambda}A^\top L  A\bigg)\\
    & \Gamma_{2_k} = \Sigma_2z^k-\frac{1}{\lambda}B^\top L  (Ax^{k+1}-c+v^k)&\\
    &\Sigma_2 = \bigg(\lambda_z I-\frac{1}{\lambda}B^\top L  B\bigg).   
    \end{alignat}
\end{subequations}
and $\text{prox}_{g}^{\varepsilon_g}(.)$, $\text{prox}_{g}^{\varepsilon_g}(.)$ are the approximate proximal operators with respect to $g$ and $h$, respectively (see Definition \ref{Def:eProx}).

We noticed that the errors considered in most analyses of inexact ADMM variants so far are due to early termination of inner iterations (truncation errors); and therefore, deterministic by design. In this work, we analyse both the deterministic and the probabilistic scenarios to account for both truncation error, loop perforation error arising from approximate computing, as well as round-off error that may arise from finite precision representations and operations. Closed-form mathematical models of practical computational error sequences are hardly available. In practical scenarios, the only available information is the error upper bound, which motivates the use the proposed probablistic analysis since they only depend on this boundedness property in addition to the relaxed conditional mean independence property between error sequences that will be explained later. The probabilistic analysis is generic and can be extended to any algorithm with bounded approximation errors.

Regarding the termination criteria, we adopt the same stopping criteria of \cite{boyd2011distributed} but with minimum absolute and relative feasibility tolerances for longer simulation period.

\subsection{Applications}
ADMM-type algorithms have a wide spectrum of applications in machine learning \cite{dhar2015admm,ding2019stochastic,liu2020admm, holmes2021nxmtransformer}, artificial intelligence \cite{yang2018admm,li2019admm}, MIMO detection \cite{un2019deep}, image reconstruction \cite{yang2018admm}, compressed sensing \cite{un2019deep} and model predictive control \cite{zhang2017embedded,darup2019towards,peccin2019fast, cheng2020semi, rey2016admm, tang2019distributed}. In this paper, we apply the proposed algorithm to a synthesized LASSO problem using randomly generated data and to a synthesized $k$-support-norm regularized robust regression using real solver's inaccuracies combined with algorithmic loop approximation errors. In order to validate our results, we consider the injected truncated Gaussian error sequences (for LASSO) as well as the unstructured real errors from solver and loop perforation inacurracies (for $k$-support-norm regularized robust regression) at the $x$ and $z$-subproblems of the WLM-ADMM scheme at every iteration and use the derived error bounds to estimate the rate of convergence in the presence of errors as well as asymptotic suboptimal residuals.   
\subsection{Contributions}
Error bounds of numerical algorithms estimate the distance to the solution set of a given problem. Convergence rate analyses and iteration complexity both depend on the quality of the estimated error bounds. Assuming stochastic bounded error sequences (Assumptions \textbf{M.1} and \textbf{M.1}), we propose new bounds that are probabilistically sharper without stronger assumptions other than error boundedness; i.e., relaxing the classical (idealistic) assumption of absolute summability. Secondly, following a new line of proof, we extend the obtained results to nonconvex lower-$\mathcal{C}^2$ objective functions without further assumptions on the constraints of problem \eqref{Eq:Problem}.  Finally, we apply our proposed algorithm and verification technique to LASSO and a $k$-support-norm regularized robust regression problem.   
\subsection{Organization}
The remainder of the paper is organized as follows: Section \ref{SectionMainResults} is dedicated to the main findings. The results are experimentally verified and validated in Section \ref{SectionExpResult}. We conclude our work and suggest future improvements in Section \ref{SectionConclusion}. 
\section{MAIN RESULTS}
Before we introduce our main findings, let us list the assumptions and error models that will be used later in the main theorems. 
\label{SectionMainResults}
\subsection{Assumptions}
In this subsection we list all assumptions that we are going to refer to when deriving the main convergence results in Sections~\ref{scenario1} and~\ref{scenario2}. Note that some assumptions may only apply in specific cases.
\begin{assumption}[Assumptions on the problem]
	\label{Assum:Problem}
	\hfill    
	\medskip  
	\noindent
\begin{enumerate}[label=\textbf{P.\arabic*}]
\item The functions $g,h:\mathbb{R}^n \to \mathbb{R} \cup \{+\infty\}$ are
lower-$\mathcal{C}^2$ functions.

\item The functions $g,h:\mathbb{R}^n \to \mathbb{R} \cup \{+\infty\}$ are
  closed, proper, and convex.
\item $M_x^k$, $M_z^k$ and $L$ are positive definite matrices. We implicitly require $\lambda\lambda_x > \|A^\top L A\|$ and $\lambda\lambda_z > \|B^\top L B\|$.
\end{enumerate}
\end{assumption}
\begin{assumption}[Error models]
	\label{Assum:ErrModels}
	\hfill    
	\medskip  
	\noindent
\begin{enumerate}[label=\textbf{M.\arabic*}]
\item
\begin{align}
    r_x^k := \overline{x}^{k}-x^{k}  \\
    r_z^k := \overline{z}^{k}-z^{k} 
\end{align}
where $\overline{x}^{k}$ and $\overline{z}^{k}$ are the exact error-free iterates, $r_x^k$ and $r_z^k$ are are additive residual error vectors in ADMM updates due to proximal errors $\varepsilon_{g}$ and $\varepsilon_{h}$, respectively. We assume noise-free initial conditions $r_x^0 = r_z^0 = 0$. 
\item
\begin{align}
    x_{\Omega}^{k} = x^{k} + r_{x_\Omega}^k \\
    z_{\Omega}^{k} = z^{k} + r_{z_\Omega}^k
\end{align}
where $r_{x_\Omega}^k$ and $r_{z_\Omega}^k$ are additive residual error vectors in ADMM updates due to random proximal errors $\varepsilon_{g_\Omega}^k$ and $\varepsilon_{h_\Omega}^k$, respectively. The $\Omega$ (sample probability space) subscript refers to the random variability in the corresponding quantities.
\item
For all $k\geq 1$, the proximal errors $\epsilon_{g_\Omega}^k$ and $\epsilon_{h_\Omega}^k$ are bounded almost surely. Specifically, we have
\begin{subequations}
    \begin{alignat}{3}
        &0 \leq \epsilon_{g_\Omega}^k \leq \varepsilon_0&\\
        &0 \leq \epsilon_{h_\Omega}^k \leq \varepsilon_0&
    \end{alignat}
\end{subequations}  
both hold with probability $1$.
\item
$\epsilon_{g_\Omega}$ and $\epsilon_{h_\Omega}$ are stationary. Specifically, we have 
\begin{subequations}
    \begin{alignat}{3}
    &\mathbb{E}\bigg[\epsilon_{g_\Omega}^k\bigg] = \mathbb{E}\bigg[\epsilon_{g_\Omega}\bigg]=\text{const.}&\\
    &\mathbb{E}\bigg[\epsilon_{h_\Omega}^k\bigg]=\mathbb{E}\bigg[\epsilon_{g_\Omega}\bigg]=\text{const.}&
    \end{alignat}
\end{subequations}  
\end{enumerate}
\end{assumption}
\subsection{Scenario 1: Approximate ADMM with deterministic errors}\label{scenario1}
In this scenario, we assume additive and deterministic error models \textbf{M.1}. We present suboptimal convergence results for the nonconvex case as well as the convex case under assumptions \textbf{P.1} and \textbf{P.2}. We show that for the deterministic case, suboptimal convergence is only achieved under summability assumption \textbf{} on the proximal errors. 
\begin{theorem}[General case with additive errors]
Assume \textbf{P.1}, \textbf{P.3} and \textbf{M.1}, then for any $x$ and $z$ such that $Ax + Bz = c$, the sequence generated using the WLM-ADMM scheme \eqref{WLM-ADMMProx} with $\lambda_x=\lambda_z=1$ satisfies the following 
\begin{align}
    &\frac{1}{k+1}\sum_{i=0}^{k}f(x^{i+1},z^{i+1})-f(x,z)\notag\\
    &+\frac{1}{k+1}\sum_{i=0}^{k}\langle\frac{1}{\lambda}L  u^{i+1},Ax^{i+1}+Bz^{i+1}-(Ax+Bz)\rangle\leq\notag\\
    &\frac{1}{2(k+1)}\Big[\|x^0-x\|^2_{\Sigma_1}+\|z^0-z\|^2\Big]\notag\\
    &+\frac{1}{k+1}\Big[\sum_{i=0}^{k}\varepsilon_g^{i+1}+\sum_{i=0}^{k}\varepsilon_h^{i+1}-\langle \Sigma_1(r_x^{k+1}-r_x),x^{k+1}-x \rangle\notag\\&-\langle r_z^{k+1}-r_z,z^{k+1}-z \rangle\Big]
    \label{Eq:Thm1}
\end{align}
For $(x^{\text{ref}},z^{\text{ref}}) = (x^\star,z^\star)$ we have
\begin{align}
    &\frac{1}{k+1}\sum_{i=0}^{k}f(x^{i+1},z^{i+1})-f(x^\star,z^\star)\\
    &+\frac{1}{k+1}\sum_{i=0}^{k}\langle\frac{1}{\lambda}L  u^{i+1},v^{k+1}-v^k\rangle\leq\\
    &\frac{1}{2(k+1)}\Big[\|x^0-x^\star\|^2_{\Sigma_1}+\|z^0-z^\star\|^2\Big]\\
    &+\frac{1}{k+1}\Big[\sum_{i=0}^{k}\varepsilon_g^{i+1}+\sum_{i=0}^{k}\varepsilon_h^{i+1}-\langle \Sigma_1r_x^{k+1},x^{k+1}-x^\star \rangle\\&-\langle r_z^{k+1},z^{k+1}-z^\star \rangle\Big]
\end{align}
where we have used the fact $r_x^{\star}=r_z^{\star}=0$.
\end{theorem}
\begin{proof}
From Lemma~\ref{lemma:Prox2SubGrad} we have
\begin{align}
    &\Gamma_{1_k}-x^{k+1} \in { \partial g}(x^{k+1})\\
    &\Gamma_{2_k}-z^{k+1} \in  { \partial h}(z^{k+1})\\
    &v^{k+1} = v^k+(Ax^{k+1}+Bz^{k+1}-c)
\end{align}
and from Definition~\ref{Def:Subgradient} the following inequalities hold:
\begin{align}
\label{Ineq:subgrad_g_h}
    &g(x^{k+1})-g(x)\leq\langle\Gamma_{1_k}-x^{k+1},x^{k+1}-x\rangle\\
    &h(z^{k+1})-h(z)\leq\langle\Gamma_{2_k}-z^{k+1},z^{k+1}-z\rangle\\
    &v^{k+1} = v^k+(Ax^{k+1}+Bz^{k+1}-c)
\end{align}
Using Definition~\ref{Def:e-Subgradient}, \eqref{Ineq:subgrad_g_h} can be extended to the approximate WLM-ADMM algorithm \label{WLM-ADMMProx2} as follows  
\begin{align}
\label{Ineq:e-subgrad_g_h}
    &g(x^{k+1})-g(x)\leq\langle\Gamma_{1_k}-x^{k+1},x^{k+1}-x\rangle+\varepsilon_g^{k+1}\\
    &h(z^{k+1})-h(z)\leq\langle\Gamma_{2_k}-z^{k+1},z^{k+1}-z\rangle +\varepsilon_h^{k+1}\\
    &v^{k+1} = v^k+(Ax^{k+1}+Bz^{k+1}-c)
\end{align}
Expanding $\Gamma_{1_k}$ and $\Gamma_{2_k}$ and using $\Sigma_1 = \big(I-\frac{1}{\lambda}A^\top L  A\big)$ and $\Sigma_2 = \big(I-\frac{1}{\lambda}B^\top L  B\big)$ we obtain
\begin{align}
\label{Ineq:e-subgrad_g_h}
    &g(x^{k+1})-g(x)\leq\langle\Sigma_1(x^k-x^{k+1}),x^{k+1}-x\rangle\\
    &-\langle\frac{1}{\lambda}A^\top L  u^{k+1},x^{k+1}-x\rangle+\varepsilon_g^{k+1}\\
    &h(z^{k+1})-h(z)\leq\langle\Sigma_2(z^k-z^{k+1}),z^{k+1}-z\rangle\\
    &-\langle\frac{1}{\lambda}B^\top L  v^{k+1},z^{k+1}-z\rangle+\varepsilon_h^{k+1}\\
    &v^{k+1} = v^k+(Ax^{k+1}+Bz^{k+1}-c)
\end{align}
where $u^{k+1} = v^{k+1}+B(z^k-z^{k+1})$. Adding both sides of the last inequalities we obtain 
\begin{align}
\label{Ineq:e-subgrad_g_h}
    &f(x^{k+1},z^{k+1})-f(x,z)\leq\langle\Sigma_1(x^k-x^{k+1}),x^{k+1}-x\rangle\\
    &+\langle\Sigma_2(z^k-z^{k+1}),z^{k+1}-z\rangle+\varepsilon_g^{k+1}+\varepsilon_h^{k+1}\\
    &-\langle\frac{1}{\lambda}L  u^{k+1},Ax^{k+1}+Bz^{k+1}-(Ax+Bz)\rangle\\
    &+\langle\frac{1}{\lambda}B^\top L  B(z^k-z^{k+1}),z^{k+1}-z\rangle\\
\end{align}
Substituting for $\Sigma_2$ with $\big(I-\frac{1}{\lambda}B^\top L  B\big)$ yields
\begin{align}
\label{Ineq:e-subgrad_g_h}
    &f(x^{k+1},z^{k+1})-f(x,z)\leq\langle\Sigma_1(x^k-x^{k+1}),x^{k+1}-x\rangle\\
    &+\langle z^k-z^{k+1},z^{k+1}-z\rangle+\varepsilon_g^{k+1}+\varepsilon_h^{k+1}\\
    &-\langle\frac{1}{\lambda}L  u^{k+1},Ax^{k+1}+Bz^{k+1}-(Ax+Bz)\rangle
\end{align}
Adding error terms from inexact iterates and using error model \textbf{M.1}
\begin{align}
\label{Ineq:e-subgrad_g_h}
    &f(x^{k+1},z^{k+1})-f(x,z)\\
    &+\langle\frac{1}{\lambda}L  u^{k+1},Ax^{k+1}+Bz^{k+1}-(Ax+Bz)\rangle\leq\\
    &\frac{1}{2}\|x^k-x\|^2_{\Sigma_1}-\frac{1}{2}\|x^{k+1}-x\|^2_{\Sigma_1}\\
    &+\frac{1}{2}\|z^k-z\|^2-\frac{1}{2}\|z^{k+1}-z\|^2+\varepsilon_g^{k+1}+\varepsilon_h^{k+1}\\
    &+\langle\Sigma_1 r_x^k,x^k-x \rangle+\langle r_z^k,z^k-z \rangle-\langle \Sigma_1r_x^{k+1},x^{k+1}-x \rangle\\
    &-\langle r_z^{k+1},z^{k+1}-z \rangle\\
    &+\frac{1}{2}\big(\|\Sigma_1 r_x^{k}\|^2+\|r_z^{k}\|^2-\|\Sigma_1 r_x^{k+1}\|^2-\|r_z^{k+1}\|^2\big)
\end{align}
Summing both sides of the inequality from $0$ to $k$ yields
\begin{align}
\label{Ineq:e-subgrad_g_h}
    &\sum_{i=0}^{k}f(x^{i+1},z^{i+1})-(k+1)f(x,z)\\
    &+\sum_{i=0}^{k}\langle\frac{1}{\lambda}L  u^{i+1},Ax^{i+1}+Bz^{i+1}-(Ax+Bz)\rangle\leq\\
    &\frac{1}{2}\|x^0-x\|^2_{\Sigma_1}-\frac{1}{2}\|x^{k+1}-x\|^2_{\Sigma_1}\\
    &+\frac{1}{2}\|z^0-z\|^2-\frac{1}{2}\|z^{k+1}-z\|^2+\sum_{i=0}^{k}\varepsilon_g^{i+1}+\sum_{i=0}^{k}\varepsilon_h^{i+1}\\
    &+\langle\Sigma_1 r_x^0,x^0-x \rangle+\langle r_z^0,z^0-z \rangle-\langle \Sigma_1r_x^{k+1},x^{k+1}-x \rangle\\
    &-\langle r_z^{k+1},z^{k+1}-z \rangle\\
    &+\frac{1}{2}\big(\|\Sigma_1 r_x^{0}\|^2+\|r_z^{0}\|^2-\|\Sigma_1 r_x^{k+1}\|^2-\|r_z^{k+1}\|^2\big)
\end{align}
where we have chosen $L = L = constant$. Using the initial conditions of error model \textbf{M.1}; i.e., $r_x^0 = r_z^0 = 0$, we obtain 
\begin{align}
\label{Ineq:e-subgrad_g_h}
    &\sum_{i=0}^{k}f(x^{i+1},z^{i+1})-(k+1)f(x,z)\\
    &+\sum_{i=0}^{k}\langle\frac{1}{\lambda}L  u^{i+1},Ax^{i+1}+Bz^{i+1}-(Ax+Bz)\rangle\leq\\
    &\frac{1}{2}\Big[\|x^0-x\|^2_{\Sigma_1}+\|z^0-z\|^2-\|x^{k+1}-x\|^2_{\Sigma_1}-\|z^{k+1}-z\|^2\Big]\\
    &+\sum_{i=0}^{k}\varepsilon_g^{i+1}+\sum_{i=0}^{k}\varepsilon_h^{i+1}-\frac{1}{2}\big(\|\Sigma_1 r_x^{k+1}\|^2+\|r_z^{k+1}\|^2\big)\\
    &-\langle \Sigma_1r_x^{k+1},x^{k+1}-x \rangle-\langle r_z^{k+1},z^{k+1}-z \rangle
\end{align}
Dividing both sides by $k+1$ completes the proof.
\begin{align}
\label{Ineq:e-subgrad_g_h}
    &\frac{1}{k+1}\sum_{i=0}^{k}f(x^{i+1},z^{i+1})-f(x,z)\\
    &+\frac{1}{k+1}\sum_{i=0}^{k}\langle\frac{1}{\lambda}L  u^{i+1},Ax^{i+1}+Bz^{i+1}-(Ax+Bz)\rangle\leq\\
    &\frac{1}{2(k+1)}\Big[\|x^0-x\|^2_{\Sigma_1}+\|z^0-z\|^2\Big]\\
    &+\frac{1}{k+1}\Big[\sum_{i=0}^{k}\varepsilon_g^{i+1}+\sum_{i=0}^{k}\varepsilon_h^{i+1}-\langle \Sigma_1r_x^{k+1},x^{k+1}-x \rangle\\&-\langle r_z^{k+1},z^{k+1}-z \rangle\Big]
\end{align}
where we have dropped negative terms from the right hand side.
\end{proof}
\begin{theorem}[Deterministic bounds for the nonconvex case]
\label{Theorem2}
Assume \textbf{P.1}, \textbf{P.3} and \textbf{M.1}, then the sequence generated using the WLM-ADMM scheme \eqref{WLM-ADMMProx} with $\lambda_x=\lambda_z=1$ and $L = L$ satisfies the following
\begin{align}
    &\frac{1}{k+1}\sum_{i=0}^{k}f(x^{i+1},z^{i+1})-f(x^\star,z^\star)\\
    &+\frac{1}{k+1}\sum_{i=0}^{k}\langle\frac{1}{\lambda}L  u^{i+1},v^{k+1}-v^k\rangle\leq\\
    &\frac{1}{2(k+1)}\Big[\|x^0-x^\star\|^2_{\Sigma_1}+\|z^0-z^\star\|^2\Big]\\
    &+\frac{1}{k+1}\Big[\sum_{i=0}^{k}\varepsilon_g^{i+1}+\sum_{i=0}^{k}\varepsilon_h^{i+1}+\|\Sigma_1 r_x^{k+1}\|\|x^{k+1}-x^\star\|\\&+\| r_z^{k+1}\|\|z^{k+1}-z^\star\|\Big]
\end{align}
\begin{proof}
Applying Cauchy-Schwarz to \eqref{Eq:Thm1} with $(x^{\text{ref}},z^{\text{ref}}) = (x^\star,z^\star)$ yields
\begin{align}
    &\frac{1}{k+1}\sum_{i=0}^{k}f(x^{i+1},z^{i+1})-f(x^\star,z^\star)\\
    &+\frac{1}{k+1}\sum_{i=0}^{k}\langle\frac{1}{\lambda}L  u^{i+1},v^{k+1}-v^k\rangle\leq\\
    &\frac{1}{2(k+1)}\Big[\|x^0-x^\star\|^2_{\Sigma_1}+\|z^0-z^\star\|^2\Big]\\
    &+\frac{1}{k+1}\Big[\sum_{i=0}^{k}\varepsilon_g^{i+1}+\sum_{i=0}^{k}\varepsilon_h^{i+1}+\|\Sigma_1 r_x^{k+1}\|\|x^{k+1}-x^\star\|\\&+\| r_z^{k+1}\|\|z^{k+1}-z^\star\|\Big]
\end{align}
which completes the proof.
\end{proof}
\end{theorem}
\begin{corollary}[Deterministic bounds for the convex case]
\label{Corol2}
Assume \textbf{P.2}, \textbf{P.3} and \textbf{M.1}, then the sequence generated using the WLM-ADMM scheme \eqref{WLM-ADMMProx} with $\lambda_x=\lambda_z=1$ and $L = L$ satisfies the following
\begin{align}
    &f\bigg(\frac{1}{k+1}\sum_{i=0}^{k}x^{i+1},\frac{1}{k+1}\sum_{i=0}^{k}z^{i+1}\bigg)-f(x^\star,z^\star)\\
    &+\frac{1}{k+1}\sum_{i=0}^{k}\langle\frac{1}{\lambda}L  u^{i+1},v^{k+1}-v^k\rangle\leq\\
    &\frac{1}{2(k+1)}\Big[\|x^0-x^\star\|^2_{\Sigma_1}+\|z^0-z^\star\|^2\Big]\\
    &+\frac{1}{k+1}\Big[\sum_{i=0}^{k}\varepsilon_g^{i+1}+\sum_{i=0}^{k}\varepsilon_h^{i+1}+\sqrt{2\varepsilon_h^{k+1}}\|z^{k+1}-z^\star\|\\&+\sqrt{2\lambda_{\max} (\Sigma_1^\top\Sigma_1)\varepsilon_g^{k+1}}\|x^{k+1}-x^\star\|\Big]
\end{align}
\end{corollary}
\begin{proof}
Applyimg Lemma~\ref{Lem:BackwardProxError} yields 
\begin{equation}
    \|\Sigma_1 r_x^{k+1}\|\leq\|\Sigma_1\|\sqrt{2\varepsilon_g^{k+1}}\leq\sqrt{2\lambda_{\max} (\Sigma_1^\top\Sigma_1)\varepsilon_g^{k+1}},
\end{equation}
and
\begin{equation}
    \|r_z^{k+1}\|\leq\sqrt{2\varepsilon_g^{k+1}}.
\end{equation}
Using these inequalities we obtain
\begin{align}
    &\frac{1}{k+1}\sum_{i=0}^{k}f(x^{i+1},z^{i+1})-f(x^\star,z^\star)
    \\
    &+\frac{1}{k+1}\sum_{i=0}^{k}\langle\frac{1}{\lambda}L  u^{i+1},v^{k+1}-v^k\rangle\leq
    \\
    &\frac{1}{2(k+1)}\Big[\|x^0-x^\star\|^2_{\Sigma_1}+\|z^0-z^\star\|^2\Big]
    \\
    &+\frac{1}{k+1}\Big[\sum_{i=0}^{k}\varepsilon_g^{i+1}+\sum_{i=0}^{k}\varepsilon_h^{i+1}+\sqrt{2\varepsilon_h^{k+1}}\|z^{k+1}-z^\star\|
    \\
    &+\sqrt{2\lambda_{\max} (\Sigma_1^\top\Sigma_1)\varepsilon_g^{k+1}}\|x^{k+1}-x^\star\|\Big]
\end{align}
Using Jensen's inequality (Lemma~\ref{lem:jensen}) completes the proof.
\end{proof}
\subsection{Scenario 2: Approximate ADMM with random errors}\label{scenario2}
In this scenario, we assume additive random error models \textbf{M.2}, \textbf{M.3} and \textbf{M.4}. We present probabilistic suboptimal convergence results for the nonconvex case as well as the convex case under assumptions \textbf{P.1} and \textbf{P.2}. For the convex case, we show that suboptimal convergence is only guaranteed under boundedness assumption on the proximal errors \textbf{M.3}. For the nonconvex case, additional boundedness assumption on the residuals is required to guarantee suboptimal convergence. 

The next two results will be proved together after Corollary~\ref{Corol2}. 
\begin{theorem}[Probabilistic bounds for the nonconvex case]
\label{Theorem3}
Assume \textbf{P.1}, \textbf{P.3}, \textbf{M.3} and \textbf{M.4}, then the sequence generated using the WLM-ADMM scheme \eqref{WLM-ADMMProx} with $\lambda_x=\lambda_z=1$ and $L = L$ satisfies the following
\begin{align}
    &\frac{1}{k+1}\sum_{i=0}^{k}f(x^{i+1},z^{i+1})-f(x^\star,z^\star)\\
    &+\frac{1}{k+1}\sum_{i=0}^{k}\langle\frac{1}{\lambda}L  u^{i+1},v^{k+1}-v^k\rangle\leq\mathbb{E}[\varepsilon_{g_\Omega}]+\mathbb{E}[\varepsilon_{h_\Omega}]\\
    &+\frac{1}{2(k+1)}\Big[\|x^0-x^\star\|^2_{\Sigma_1}+\|z^0-z^\star\|^2+\|r_z^{k+1}\|\|z^{k+1}-z^\star\|\\&+\|\Sigma_1 r_x^{k+1}\|\|x^{k+1}-x^\star\|\Big]+\frac{\gamma}{\sqrt{k+1}}\varepsilon_0
\end{align}
holds with probability at least $1 - 4\exp(-\frac{\gamma^2}{2})$ for any $\gamma >0$.
\end{theorem}
\begin{corollary}[Probabilistic bounds for the convex case]
Assume \textbf{P.2}, \textbf{P.3}, \textbf{M.3} and \textbf{M.4},  then the following holds
\label{Corol3}
\begin{align}
    &f\bigg(\frac{1}{k+1}\sum_{i=0}^{k}x^{i+1},\frac{1}{k+1}\sum_{i=0}^{k}z^{i+1}\bigg)-f(x^\star,z^\star)\\
    &+\frac{1}{k+1}\sum_{i=0}^{k}\langle\frac{1}{\lambda}L  u^{i+1},v^{k+1}-v^k\rangle\leq\mathbb{E}[\varepsilon_{g_\Omega}]+\mathbb{E}[\varepsilon_{h_\Omega}]\\
    &+\frac{1}{2(k+1)}\Big[\|x^0-x^\star\|^2_{\Sigma_1}+\|z^0-z^\star\|^2+2\sqrt{2\varepsilon_h^{k+1}}\|z^{k+1}-z^\star\|\\&+2\sqrt{2\lambda_{\max} (\Sigma_1^\top\Sigma_1)\varepsilon_g^{k+1}}\|x^{k+1}-x^\star\|\Big]\\ &+\frac{\gamma}{\sqrt{k+1}}\varepsilon_0
\end{align}
holds with probability at least $1 - 4\exp(-\frac{\gamma^2}{2})$ for any $\gamma >0$.
\end{corollary}
\begin{proof}
Here $\epsilon_{g_\Omega}$ and $\epsilon_{h_\Omega}$ are bounded almost surely and have stationary means. Specifically, we have $0 \leq \epsilon_{g_\Omega}^k \leq \varepsilon_0$ and $0 \leq \epsilon_{h_\Omega}^k \leq \varepsilon_0$ both hold with probability $1$.  
Applying Lemma \ref{lem:4} to both error sequences, we can write,
\begin{equation}
\begin{split}
    \label{Ineq:HoeffBeforeStationarity}
    &\text{Pr}\Bigg(\Bigg|\sum_{i=0}^{k}\epsilon_{(.)_\Omega}^{i+1} - \mathbb{E}\Bigg[\sum_{i=0}^{k}\epsilon_{(.)_\Omega}^{i+1}\Bigg]\Bigg|\geq t\Bigg)\leq 2\exp\bigg(\frac{-2t^2}{(k+1)\varepsilon_0^2}\bigg),\quad \\&\text{for all}\quad t > 0.
\end{split}
\end{equation}
where $Pr$ is a probability measure in the sample space $\Omega$, $(.)$ stands for either $g$ or $h$.

Defining the constant mean  $\mathbb{E}\big[\epsilon_{(.)_\Omega}^k\big] = \mathbb{E}\big[\epsilon_{(.)_\Omega}\big]$ and substituting in \eqref{Ineq:HoeffBeforeStationarity} yields 
\begin{equation}
\begin{split}
    \label{Ineq:HoeffAfterStationarity}
    &\text{Pr}\Bigg(\bigg|\sum_{i=0}^{k}\epsilon_{(.)_\Omega}^{i+1} - k\mathbb{E}\big[\epsilon_{(.)_\Omega}\big]\bigg|\geq t\Bigg)\leq 2\exp\bigg(\frac{-2t^2}{(k+1)\varepsilon_0^2}\bigg),
    \quad
    \\&\text{for all}\quad t > 0.  
\end{split}
\end{equation}
By choosing $t = \frac{\gamma\sqrt{k+1} \varepsilon_0}{2}$, for some $\gamma > 0$, we obtain
\begin{equation}
    \begin{split}
    &\text{Pr}\Bigg(\bigg|\sum_{i=0}^{k}\epsilon_{(.)_\Omega}^i - k\mathbb{E}\big[\epsilon_{(.)_\Omega}\big]\bigg|\geq \frac{\gamma\sqrt{k+1} \varepsilon_0}{2}\Bigg)\leq 2\exp\bigg(\frac{-\gamma^2}{2}\bigg),
    \quad
    \\&\text{for all}\quad \gamma > 0.
    \end{split}
\end{equation}
Equivalently,
\begin{equation}
    \sum_{i=0}^{k}\epsilon_{(.)_\Omega}^i \leq (k+1)\mathbb{E}\big[\epsilon_{(.)_\Omega}\big] + \frac{\gamma\sqrt{k+1} \varepsilon_0}{2}
    \label{e2leq}
\end{equation}
holds with probability at least $1 - 2\exp(-\frac{\gamma^2}{2})$ for any $\gamma >0$. 

Using the last inequality in Theorem~\ref{Theorem2} and Corollary~\ref{Corol2} and applying Lemma~\ref{lem:sumprob} completes the proof of Theorem~\ref{Theorem3} and Corollary~\ref{Corol3}, respectively. 
\end{proof}
\section{EXPERIMENTAL RESULTS}
\label{SectionExpResult}
In this section, we validate our proposed technique on regularized inverse problems using the WLM-ADMM scheme with simulated and practical computational inaccuracies. 
\subsection{Experimental setup}
\label{SectionExperiment}

Synthetic data matrix $A\in\mathbb{R}^{m\times n}$ and vector $b\in\mathbb{R}^m$ are randomly generated with $m=500$ and $n=100$. 

To simplify the algorithm verification process, the WL-ADMM scheme \eqref{WLM-ADMMProx} is used with $L=I$, $\rho=1.2$ and $\lambda_x=\lambda_z=1$ in both experiments. The algorithm is also randomly intialized for each experiment.

In the LASSO experiment, we inject simulated computational errors generated from a zero-mean truncated Gaussian distribution. The parameter $\delta$ is used as a noise level parameter that controls the range of the generated error sequences. $\delta \approx 3\sigma$, where $\sigma$ is the standard deviation. The generated error sequences satisfy $|\varepsilon_g^k|\leq\delta|  x^k|$, $|\varepsilon_h^k|\leq\delta|  z^k|$ for all $k\geq1$, and for $k=0$ we have $\varepsilon_g^0 = \varepsilon_h^0 =0$. 

In the regularized robust regression experiment, two types of approximation errors are introduced. The first approximation error corresponds to the early termination of step \eqref{subproblem_g} which is introduced via CVX's \verb|cvx_precision| parameter. The second approximation error corresponds to loop perforation in step \eqref{subproblem_h} which is introduced via a $skip$ parameter in both inner loops of Algorithm 1 in \cite{argyriou2012sparse}. The error sequences $\varepsilon_g^k$ and $\varepsilon_h^k$ can be calculated as follows
\begin{align}
    \varepsilon_g^k &= \bigg| g(\text{prox}_g^{\varepsilon_g}(\Gamma_{1_k})) + \frac{1}{2}\|\text{prox}_g^{\varepsilon_g}(\Gamma_{1_k}) -\Gamma_{1_k}\|_2^2 \notag\\
    &-g(\text{prox}_g(\Gamma_{1_k}))- \frac{1}{2}\|\text{prox}_g(\Gamma_{1_k}) -
      \Gamma_{1_k}\|_2^2
      \bigg|
\end{align}
\begin{align}
    \varepsilon_h^k &= \bigg| h(\text{prox}_h^{\varepsilon_h}(\Gamma_{2_k})) + \frac{1}{2}\|\text{prox}_h^{\varepsilon_h}(\Gamma_{2_k}) -\Gamma_{2_k}\|_2^2 \notag\\
    &-h(\text{prox}_h(\Gamma_{2_k}))- \frac{1}{2}\|\text{prox}_h(\Gamma_{2_k}) -
      \Gamma_{2_k}\|_2^2
      \bigg|.
\end{align}

We measure the distance from the optimal set using the suboptimality metric $f(x^k,z^k)-f(x^\star, z^\star)$ which we simply refer to it as $f^k-f^\star$ in the plots. The reference $f^\star = f(x^\star, z^\star)$ is computed with high accuracy using highly precise solution for the $x$-subproblem ($\varepsilon_g \leq 2.2204\times10^{-16}$) and a skip size of $skip=1$ for the $z$-subproblem.

In order to validate our probabilistic statement, we have calculated the empirical probabilities for each experiment. The empirical probability is given by
\begin{equation}
    p = \frac{\#[(f^k-f^\star)<B(\gamma)]}{N}
\end{equation}
where $B(\gamma)$ is the $\gamma$-parameterized probabilistic bound of Theorem~\ref{Theorem3}, $\#$ is the counter operator and $N$ is the experimental maximum number of iterations $k \leq N$.

The simulation period is set to $3000$ iterations in both experiments.
\subsection{LASSO}
In this experiment, we use the approximate WLM-ADMM scheme of \eqref{WLM-ADMMProx} to solve a synthesized LASSO problem with simulated truncated Gaussian error sequences
\begin{equation}
  \label{Eq:LassoProblem}
  \begin{split}
      &\underset{x \in \mathbb{R}^n,z \in \mathbb{R}^n}{\text{minimize}} \,\,\,
      f(x) := \|Ax-b\|_2^2 + \|z\|_1,\\
      &\text{subject to} \,\,\,
      x-z = 0\,,  
  \end{split}
\end{equation}
The calculated error bounds are depicted below with their corresponding empirical probabilities. Figure \ref{Figure 1} shows the automatic adaptation of probabilistic bounds to different noise levels ($\delta=20\%$ and $\delta=20\%$).
\begin{figure}
\includegraphics[width=10cm]{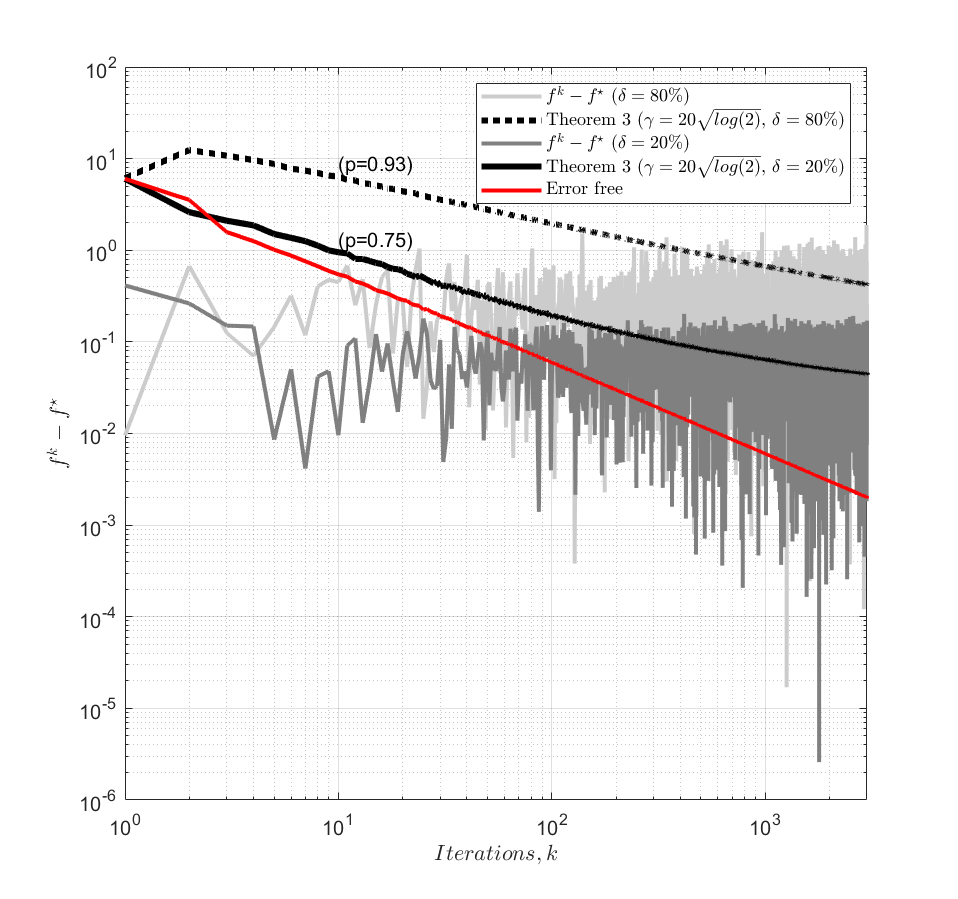}
\caption{Probabilistic upper bounds based on Theorem 3 vs error-free bound for problem \eqref{Eq:LassoProblem} (with $\lambda=1$), with different injected noise levels ($\delta = 20\%$ and $\delta = 80\%$) and a fixed probability parameter of $\gamma = 20\sqrt{\log(2)}$. The ADMM solver tolerances are fixed to $\text{RELTOL}=\text{ABSTOL}=10^{-6}$.\label{Figure 1}}
\end{figure}
In Figure \ref{Figure 2}, at a fixed noise level of $\delta=20\%$, the probabilistic bounds with $\gamma = 2\sqrt{\log(2)}$ and $\gamma = 20\sqrt{\log(2)}$ are superimposed on the error-free bound for comparison. 
\begin{figure}
\includegraphics[width=10cm]{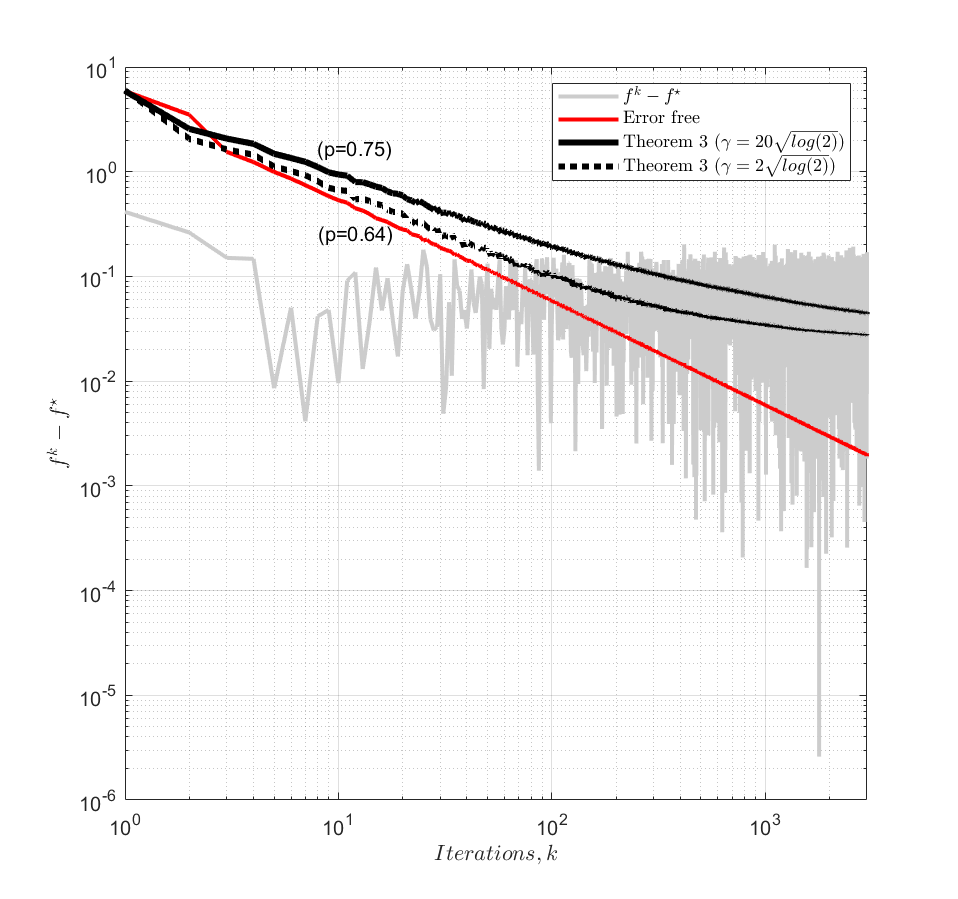}
\caption{Probabilistic upper bounds based on Theorem 3 vs error-free bound for problem \eqref{Eq:LassoProblem} (with $\lambda=1$), with an injected noise level of $\delta = 20\%$ and different values for the probability parameter ($\gamma = 2\sqrt{\log(2)}$ and $\gamma = 20\sqrt{\log(2)}$). The ADMM solver tolerances are fixed to $\text{RELTOL}=10^{-4}=\text{ABSTOL}=10^{-4}$.
\label{Figure 2}}
\end{figure}
In  both figures, the probabilistic error bounds tend to predict the rate  and suboptimal residual (in the objective function values) with high accuracy. The proposed bounds are also highly adaptive to different noise levels (Figure \ref{Figure 1}). Although we fixed the parameter $\gamma$ in Figure \ref{Figure 2}, this can be automatically tuned for better accuracy. Both figures suggest that there is a trade-off between rate and residual estimation, and a small $\gamma$ always gives better rate estimates at the expense of degraded residual prediction in the presence of parasitic computational noise. 

The following experiment considers a more practical situation, where the computational noise comes from a combination of software tolerance and algorithmic approximation.

\newpage 

\subsection{Robust Regression with $k$-Support Norm Regularization}
In this experiment, we use the WL-ADMM scheme of \eqref{WLM-ADMMProx} to solve a synthesized robust regression problem with $k$-support norm regularization
\label{SectionExperiment}
\begin{equation}
  \label{problem_l1_ksupp}
  \begin{split}
      &\underset{x \in \mathbb{R}^n,z \in \mathbb{R}^m}{\text{minimize}} \,\,\,
      f(x) := \frac{1}{2}\|Ax-b\|_1 + \frac{\lambda}{2}(\|z\|_{k}^{k_{\text{supp}}})^2,\\
      &\text{subject to} \,\,\,
      x-z = 0\,,  
  \end{split}
\end{equation}
with $\lambda=1$ and $k_{\text{supp}}=20$.
Figure \ref{Figure 4} below shows the performance of the WL-ADMM algorithm and the accuracy of the proposed bounds in the absence of computational noise. The error-free bound and the proposed probabilistic bound coincide. As predicted by both convergence bounds, the WL-ADMM algorithm converges to the optimal solution with high accuracy at a linear iteration rate complexity $O(1/k)$.

\begin{figure}
\includegraphics[width=10cm]{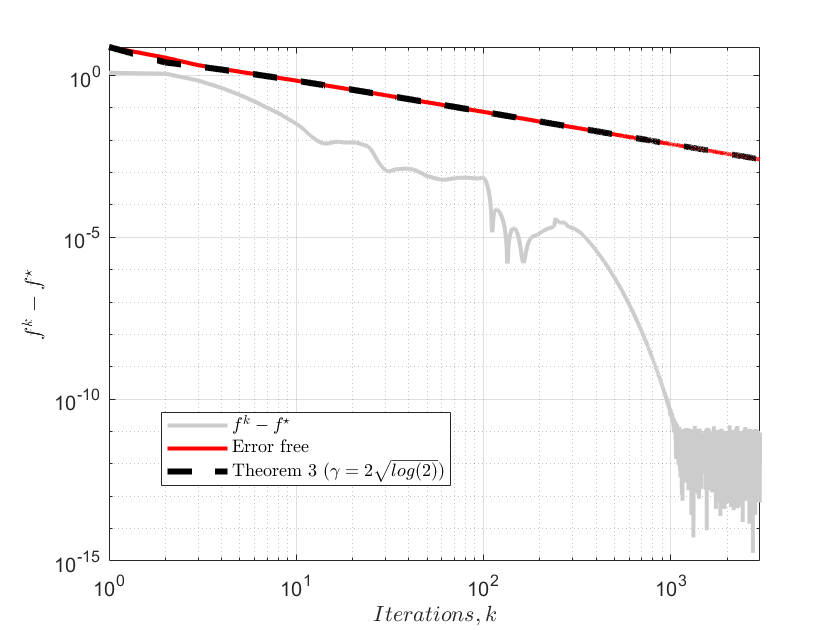}
\caption{Probabilistic upper bounds based on Theorem 3 vs error-free bound for problem \eqref{problem_l1_ksupp} (with $k_{\text{supp}}=20$ and $\lambda=1$). This is obtained using the exact solver without any loop perforation, i.e., $skip=1$; and the ADMM solver tolerances are fixed to $\text{RELTOL}=\text{ABSTOL}=2.2204\times10^{-16}$.}
\label{Figure 4}
\end{figure}


For the approximate WL-ADMM, we use different loop perforation $skip$ sizes in \eqref{subproblem_h}. The resulting $\varepsilon_g$ and $\varepsilon_h$ proximal error sequences of ADMM's subproblems \eqref{subproblem_g} and \eqref{subproblem_h} are plotted below 

\begin{figure}[H]
\includegraphics[width=10cm]{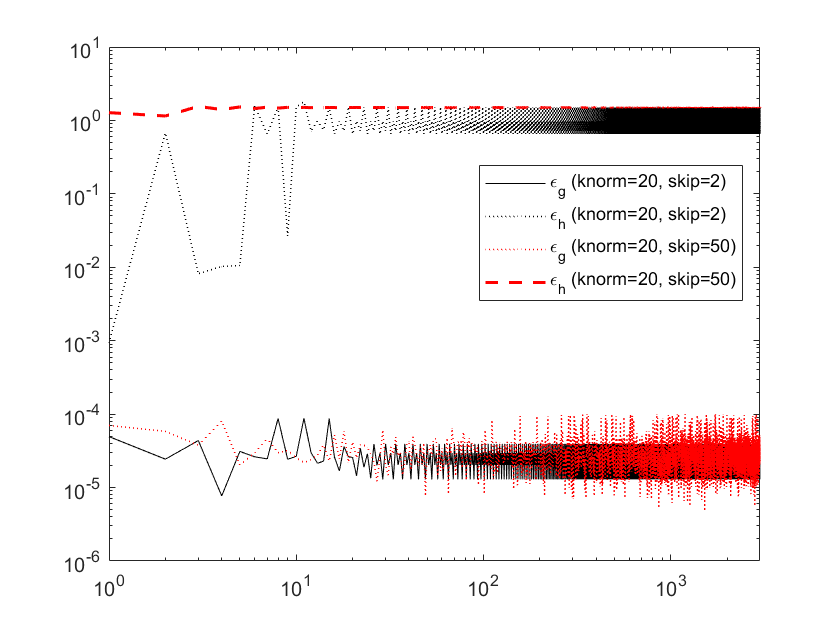}
\caption{ADMM subproblem's inner loop solver tolerances ($\varepsilon_g$ and $\varepsilon_h$) for $k_{\text{supp}}=20$ and different loop perforation sizes ($skip=2$ and $skip=50$). The outer loop ADMM solver tolerances are fixed to $\text{RELTOL}=\text{ABSTOL}=2.2204\times10^{-16}$.}
\label{Figure 3}
\end{figure}

We run another experiment by increasing the loop perforation skip size to $skip=2$ in both inner loops of Algorithm 1 in \cite{argyriou2012sparse}. The probabilistic bound with $\gamma=2\sqrt{\log(2)}$ is superimposed on the error-free bound for comparison in Figure~\ref{Figure 5}. We can see that the probabilistic bound, in solid black, tracks the residual error with high probability ($p=1.00$) while the error-free bound, in red, shows no correlation with computational noise. The empirically calculated probability conforms with Theorem~\ref{Theorem3}, which states that the bound holds with $p \geq 1-4\exp({-\gamma^2/2})$.  

\begin{figure}
\includegraphics[width=10cm]{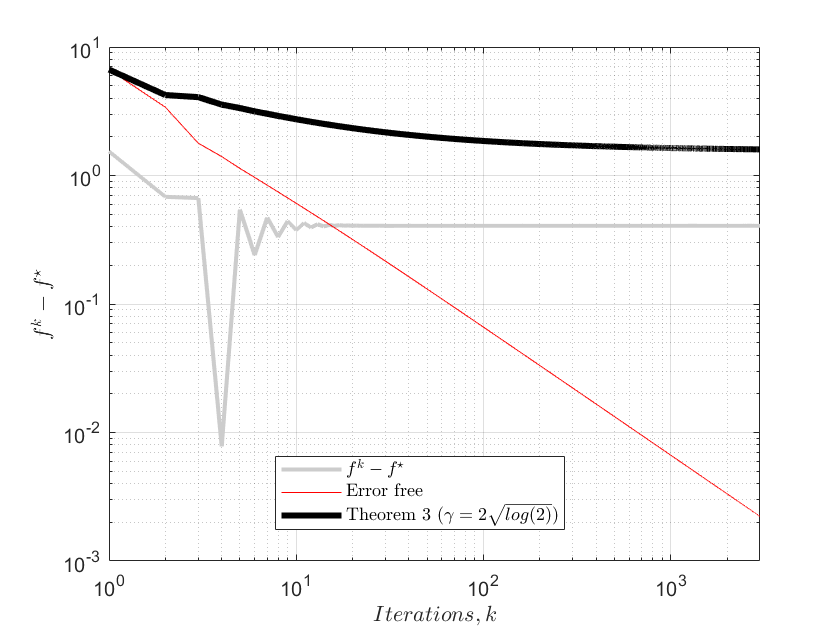}
\caption{Probabilistic upper bounds based on Theorem 3 vs error-free bound for problem \eqref{problem_l1_ksupp} (with $k_{\text{supp}}=20$ and $\lambda=1$). The loop perforation size is $skip=2$; and the ADMM solver tolerances are fixed to $\text{RELTOL}=\text{ABSTOL}=2.2204\times10^{-16}$.}
\label{Figure 5}
\end{figure}


\begin{figure}
\includegraphics[width=10cm]{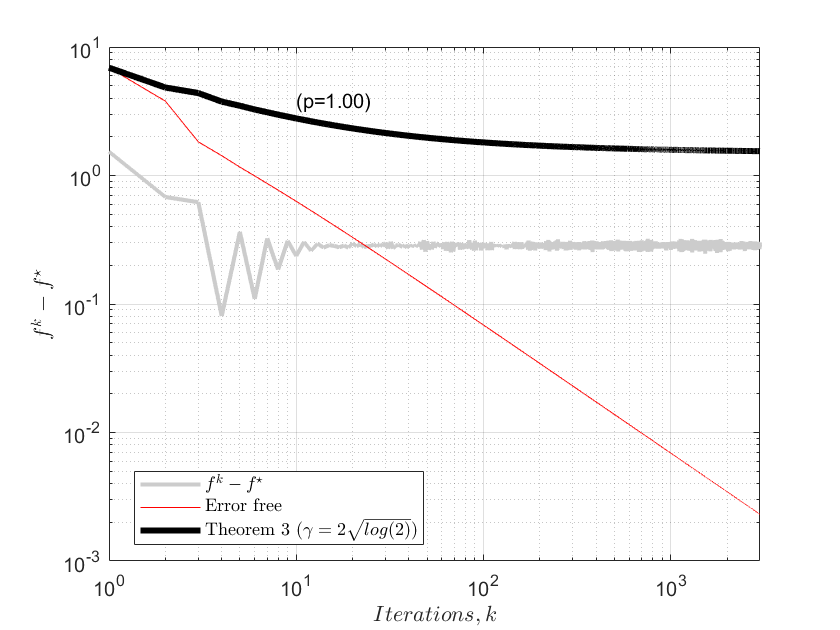}
\caption{Probabilistic upper bounds based on Theorem 3 vs error-free bound for problem \eqref{problem_l1_ksupp} (with $k_{\text{supp}}=20$ and $\lambda=1$). The loop perforation size is $skip=50$; and the ADMM solver tolerances are fixed to $\text{RELTOL}=\text{ABSTOL}=2.2204\times10^{-16}$.}
\label{Figure 6}
\end{figure}

In order to emphasize the accuracy of the proposed probabilistic bounds in suboptimal convergence prediction, we run another experiment but with increasing the loop perforation skip size to $skip = 50$ which corresponds to increased proximal computational errors. Figure~\ref{Figure 3} depicts the increased noise level in both $\varepsilon_g$ and $\varepsilon_h$ proximal errors. The function value iterations and the corresponding error-free and probabilistic convergence bounds are plotted in Figure \ref{Figure 6}. As the number of iterations increases, the probabilistic bound becomes tighter and more accurate in estimating the noise ball around the optimum as well as the rate of convergence. The calculated probability also agrees with the lower bound of Theorem~\ref{Theorem3}.

Although the theoretical lower bound on the probability is $1-4\exp(-\gamma^2/2)$, we have found that $1-2\exp(-\gamma^2/2)$ is empirically tighter in all practical experiments.


\section{CONCLUSION}
\label{SectionConclusion}
In this work, we analysed the convergence of a fully inexact and generalized ADMM version which we referred to as WL-ADMM. We derived general upper bounds on the algorithm's iterations truncation error assuming additive and deterministic approximation error models. The proposed convergence results apply to both the nonconvex case as well as to convex case under the Lipschitz continuity condition. We also found that for the deterministic case, the WL-ADMM suboptimal convergence is only guaranteed under summability assumption on the approximation error sequences of the ADMM's subproblems iterations. 

We found that the summability assumption can be relaxed by using the same additive but also bounded and random error sequences. We have established probabilistic convergence results for the lower-$\mathcal{C}^2$ nonconvex case as well as the Lipschitz continuous convex case under error boundedness and conditional mean independence conditions.

The proposed probabilistic bounds were validated and tested on practical software's (CVX solver) early termination errors combined with loop perforation error of different sizes. We also tested the proposed bounds on injected simulated noise. Based on the presented experimental results from synthetic LASSO and k-support norm regularized robust regression problems, the probabilistic bounds predicted the residual error and the convergence rate (to the suboptimal solution) with high accuracy without any explicit or implicit assumptions on the approximation errors statistics nor its dynamics (vanishing rate).
\appendix
\section*{Some known results}
\begin{lemma}
\label{lemma:Prox2SubGrad}
Let $x \in \mathcal{H}$. Then $x^{+} = \text{prox}_{q}(x)$ if, and only if, $(x-x^{+})\in \partial q(x^{+})$.
\end{lemma}
\begin{lemma}\label{Lem:BackwardProxError} 
Given a convex function $f$. For $\lambda > 0$, define $G$: $\mathbb{R}^n \times \mathbb{R}^n \rightarrow (-\infty, \infty]$ as the proper, closed, and $1/\lambda$-strongly convex function
\begin{equation}
    G\big(y,\, x\big):=  f(y) +\frac{1}{2\lambda}\norm{x-y}_{2}^2, \notag
\end{equation}
Define $\widehat{y}^\star := \arg \min G\big(y,\, x\big)$ as the minimizer of $G$ with respect to $y$ when $x$ is fixed, and $y^\star \in \{y : G(y,x)-G(\widehat{y}^\star,x) \leq \varepsilon\}$ as an $\varepsilon$-approximate solution of the same problem. Then, 
\begin{equation}
     \big\|\widehat{y}^\star - y^\star\big\|_2 \leq \sqrt{2\lambda\varepsilon}.\notag  
\end{equation}
\end{lemma}
\begin{proof}
See \cite[p.\ 43]{hamadouche2022sharper}.
\end{proof}
\begin{lemma}[\textbf{Hoeffding bound}] \label{lem:4} 
Suppose that the random variables $X_i$, $i= 1,\dots,n$ are
independent, and $X_i$ has mean $\mu_i$ and sub-Gaussian parameter $\sigma_i$. If we define $S = \sum_{i=1}^{k}X_i$ then for all $t \geq 0$, we have
\begin{equation}
    \text{Pr}\bigg(|S - \mathbb{E}\big[S\big]|> t\bigg)\leq 2\textup{exp}\bigg(\frac{-t^2}{2\sum_{i=1}^{k}\sigma_i^2}\bigg).    
\end{equation}
In particular, if $X_i \in [a, b]$ for all $i = 1, 2,\dots,n$, then
\begin{equation} 
    \text{Pr}\bigg(|S - \mathbb{E}\big[S\big]|\geq t\bigg)\leq 2\textup{exp}\bigg(\frac{-2t^2}{k(b-a)^2}\bigg).    
\end{equation}
\end{lemma}
\begin{proof}
See \cite[p.\ 24]{wainwright2019high}.
\end{proof}
\begin{lemma}
\label{lem:sumprob}
Let $(\Omega, \mathcal{F}, Pr)$ be a probability space and $T_i, i=1,\dots, m$, events in $\mathcal{F}$. Let $t_i$ be some function of a scalar variable $\gamma$. If we have
\begin{equation} 
    \text{Pr}\bigg(T_i\geq t_i(\gamma)\bigg)\leq P_i(\gamma),    
\end{equation}
for all $i=1,\dots, m$, then the following holds
\begin{equation} 
    \text{Pr}\bigg(\cup_{i=1}^m T_i\geq t_i(\gamma)\bigg)\leq \sum_{i=1}^m P_i(\gamma).    
\end{equation}
Equivalently we have
\begin{equation} 
    \text{Pr}\bigg(\cup_{i=1}^m T_i\leq t_i(\gamma)\bigg)\geq 1-\sum_{i=1}^m P_i(\gamma).    
\end{equation}
\end{lemma}
\begin{lemma}[Jensen's inequality]
\label{lem:jensen}
For a convex function f, Jensen's inequality implies
\begin{equation*}
    f\bigg(\frac{1}{k+1}\sum_{i=0}^{k}x^{i+1}\bigg) \leq
    \frac{1}{k+1} \sum_{i=0}^{k}f(x^{i+1})\,,
\end{equation*}
\end{lemma}
\addcontentsline{toc}{section}{Acknowledgment}
\section*{Acknowledgements}
This work was supported by the Engineering and Physical Research Council (EPSRC) grants EP/T026111/1, EP/S000631/1, and the MOD University Defence Research Collaboration (UDRC).
\printbibliography
\end{document}